\documentclass[11pt,a4paper,leqno]{amsart}

\usepackage[latin1]{inputenc}
\usepackage[T1]{fontenc}
\usepackage{amsfonts}
\usepackage{amsmath}
\usepackage{amssymb}
\usepackage{eurosym}
\usepackage{mathrsfs}
\usepackage{palatino}
\usepackage{color}
\usepackage{xcolor}
\usepackage{esint}
\usepackage{url}
\usepackage{hyperref}

\newcommand{\R}{\mathbb{R}}

\newcommand{\C}{\mathbb{C}}

\newcommand{\wt}{\widetilde}
\numberwithin{equation}{section}

\newcommand{\ud}[0]{\,\mathrm{d}}

\newcommand{\dist}[0]{\operatorname{dist}}
\newcommand{\Osc}[0]{\operatorname{Osc}}

\newcommand{\abs}[1]{|#1|}

\newcommand{\Babs}[1]{\Big|#1\Big|}

\newcommand{\BMO}[0]{\operatorname{BMO}}

\newcommand{\loc}[0]{\operatorname{loc}}

\newcommand{\sign}[0]{\operatorname{sgn}}

\newcommand{\eps}[0]{\varepsilon}

\theoremstyle{plain}
\newtheorem{thm}{Theorem}[section]
\newtheorem{lem}[thm]{Lemma}
\newtheorem{prop}[thm]{Proposition}

\theoremstyle{definition}

\theoremstyle{remark}
\newtheorem{rem}[thm]{Remark}

\title{Complex median method}

\author[J. Li]{Jinsong Li}

\address[J. Li]{Center for Applied Mathematics, Tianjin University, Weijin Road 92, 300072 Tianjin, China}

\email{ljs@tju.edu.cn}
\author[K. Li]{Kangwei Li}
\address[K. Li]{School of Mathematical Sciences, Zhejiang Normal University, Jinhua 321004, China}
\email{kangwei.li@zjnu.edu.cn}

\makeatletter
\@namedef{subjclassname@2020}{%
  \textup{2020} Mathematics Subject Classification}
\makeatother


\begin{document}

\allowdisplaybreaks

\begin{abstract}
In this paper, we complete the characterization of the boundedness of higher order commutators of  pointwise multiplication and a large class of singular integral operators, when the corresponding kernels satisfy certain non-degenerate conditions. This was known with the restriction that the pointwise multiplier is real-valued. We finally remove this restriction by proposing a novel complex median method. We also showcase the flexibility of this method by proving the necessity of the rectangular product BMO for the bi-commutators. 
\end{abstract}

\maketitle
\section{introduction}
Given a locally integrable function $b$ and a singular integral operator $T$, we may formally define 
the commutator $[b, T]$ as 
\[
[b, T]f=bTf-T(bf).
\]
We are interested in studying the lower bound of commutators, i.e., assuming $[b, T]$ is bounded 
on $L^p$, can we say anything about the pointwise multiplier $b$? The same question for higher order cases--we may define inductively 
\[
C_b^k(T)=[b, C_b^{k-1}(T)],\qquad C_b^1(T):=[b, T],
\]
are also under consideration. 

This topic came into the horizons of harmonic analysts due to the celebrated work \cite{CRW}, in fact there are some follow-up work quite soon, see e.g.  \cite{J, U}.
In recent years, the so-called median method (formally raised in \cite{LOR19} and refined in \cite{Hyt18}) and the approximate weak factorization method (see \cite{Hyt18}) are the main weapons to attack this problem. Both methods have their own limitations.  The median method can be applied to the higher order case, but requires the pointwise multiplier to be real-valued, while the approximate weak factorization method relaxes the restriction on the pointwise multiplier, 
but only in the first order case. 

We aim to complete the picture of this topic, that is, we will propose a new method (we call it complex median method) which is applicable not only in higher order case, but also allows the pointwise multiplier to be complex-valued. We should mention that this is not the first time that the terminology `complex median method' appears, see \cite{WZ} and \cite{HytWu}. However, the previous version of this method only applies to first-order commutators. 

In order to know the details of our new method, we introduce some notations. We consider two-variable 
Calder\'on-Zygmund kernels under the standard conditions 
\[
|K(x,y)|\le \frac{C}{|x-y|^n},\qquad \forall x\neq y,\]
\[
|K(x,y)-K(x',y)|+|K(y,x)-K(y,x')|\le \frac{1}{|x-y|^n} \omega \left(\frac{|x-x'|}{|x-y|}\right),
\]
whenever $|x-x'|<|x-y|/2$, where the modulus of continuity $\omega:[0,1)\mapsto [0,\infty)$ is increasing. We shall refer to such a kernel as $\omega$-Calder\'on-Zygmund kernel. The classical Calder\'on-Zygmund theory demands at least $\int_0^1 \omega(t) \frac{\ud t}t<\infty$, here for our purpose we require only $\omega(t)\to 0$ as $t\to 0$. We also consider homogeneous Calder\'on-Zygmund  kernels 
\[
K(x,y)= |x-y|^{-n} \,\Omega\left(\frac{x-y}{|x-y|}\right),
\]
where $\Omega \in L^1(\mathbb S^{n-1})$. As in \cite{Hyt18}, we consider the following two types of non-degeneracy assumptions: 
\begin{enumerate}
\item $K$ is an $\omega$-Calder\'on-Zygmund kernel with $\omega(t)\to 0$ as $t\to 0$ and there exists some $c_0>0$ such that for every $y\in \R^n$ and $r>0$, there exists $x\in B(y,r)^c$ with 
\[
|K(x,y)|\ge \frac 1{c_0 r^n}.
\]
\item $K$ is a homogeneous Calder\'on-Zygmund kernel with $\Omega\in  L^1(\mathbb S^{n-1})$
and there exists a Lebesgue point $\theta_0\in \mathbb S^{n-1}$ such that $\Omega(\theta_0)\ne 0$.
\end{enumerate}The related singular integrals will be referred to as non-degenerate $\omega$-Calder\'on-Zygmund operators and non-degenerate homogeneous Calder\'on-Zygmund operators, respectively. 
A consequence of the non-degeneracy assumptions is the following 
\begin{prop}\cite[Proposition 2.2.1]{Hyt18}\label{prop:key}
Let $K$ satisfy one of the non-degeneracy assumptions as above. Then for every $A\ge 3\sqrt n$ and every cube
$Q$, there is a disjoint cube $\wt Q$ at the distance $\dist(c_Q, c_{\wt Q})\sim A\ell(Q)$ such that 
$\ell(Q)=\ell(\wt Q)$ and $|K(c_{\wt Q}, c_Q)|\sim |A\ell(Q)|^{-n}$ and for all $y_1\in Q$ and $x_1\in \wt Q$, we have 
\[
\int_Q |K(x_1,y)-K(c_{\wt Q}, c_Q)|\ud y+ \int_{\wt Q} |K(x, y_1)-K(c_{\wt Q}, c_Q)| \ud x\lesssim \eps_A A^{-n},
\]
where  $\eps_A \to 0$ as $A\to 0$ and $c_Q, c_{\wt Q}$ are the centers of $Q$ and $\wt Q$, respectively.
\end{prop}

Now we can explain the key idea of our method. For instance in the first order case  the key is to connect the mean oscillation, say 
\[
\frac 1{|Q|}\int_Q |b-b_Q|,\qquad b_Q=\frac1{|Q|}\int_Q b,
\]
with $\langle [b, T] (f_Q), g_{\wt Q}\rangle$ for suitable functions $f_Q, g_{\wt Q}$ supported on $Q$ and $\wt Q$, respectively. This is quite the same like the usual median method. However, the difference is, the usual median method considers $|Q|^{-1}\int_Q |b-\alpha|$, where $\alpha$ is the median of $b$ on $\wt Q$, then $b-\alpha$ has fixed sign in a subset of $\wt Q$ with comparable size and hence one can insert  the kernel directly.  Our new method, however, simply inserts $K(c_{\wt Q}, c_Q)$ (which is a constant!) and then uses Proposition \ref{prop:key} to pass to the kernel. 
This gives us enough flexibility to remove the restriction on $b$. 

Our first main result is stated as the following.

\begin{thm}\label{thm}Let $1<p<\infty$, $b\in L_{\loc}^k(\R^n)$ and $T$ be a non-degenerate $\omega$-Calder\'on-Zygmund operator or a  non-degenerate homogeneous Calder\'on-Zygmund operator. Then there holds that
\[
\|b\|_{\BMO}^k\lesssim \|C_b^k(T)\|_{L^p\to L^p}.
\]
\end{thm}
Here recall that 
\[
\|b\|_{\BMO}:= \sup_{Q:\, \text{cubes in $\R^n$}}\frac 1{|Q|}\int_Q |b-b_Q|.
\]
We will actually prove a stronger version of Theorem \ref{thm}--the Bloom lower bound (details will be given in Section \ref{Sec:bl}, we also send the readers to \cite{B, HLW, Hyt16, LOR19} for the backgrounds of Bloom type estimates). In some sense this already showcases the flexibility of our new method. But not only this, the readers can find that our techniques can be easily extended to multilinear and/or multi-parameter setting (see e.g. \cite{GLW,L22,LMV,LMV21,O}), and therefore removing the restriction of the pointwise multiplier should be real-valued when higher-order commutators are considered. To emphasize this point, we shall show that our method can be applied to the bi-commutator as well. 

The bi-commutator was systematically considered in \cite{AHLMO}, focusing on the necessary conditions for $b$ when $[[b, T_1], T_2]$ is bounded. Here we will focus on the case 
\[
\|[[b, T_1], T_2]\|_{L^2(\R^{n}\times \R^m)\to L^2(\R^{n}\times \R^m)}<\infty,
\] where $T_1, T_2$ are \emph{symmetrically} non-degenerate $\omega$-Calder\'on-Zygmund operators on $\R^n$ and $\R^m$, respectively. Here recall that we say $T$ is a symmetrically non-degenerate $\omega$-Calder\'on-Zygmund operator on $\R^n$ if its kernel $K$ is an $\omega$-Calder\'on-Zygmund kernel with $\omega(t)\to 0$ as $t\to 0$ and there exists some $c_0>0$ such that for every $y\in \R^n$ and $r>0$, there exists $x\in B(y,r)^c$ with 
\[
|K(x,y)|\ge \frac 1{c_0 r^n}\qquad \text{and}\qquad |K(y,x)|\ge \frac 1{c_0 r^n}.
\]
We shall provide a new proof for the necessity of rectangular product BMO membership of $b$. Strictly speaking this new proof does not really involve a 'median', but nevertheless it shares the same spirit as the proof of Theorem \ref{thm}. 

We record our second main result as the following.

\begin{thm}\label{thm:m2}
Let $T_1, T_2$ be symmetrically non-degenerate $\omega$-Calder\'on-Zygmund operators on $\R^n$ and $\R^m$, respectively. Suppose that $b\in L_{\loc}^2(\R^n \times \R^m)$ and
\[
\|[ [b, T_1], T_2]\|_{L^2(\R^{n}\times \R^m)\to L^2(\R^{n}\times \R^m)}<\infty,
\]
then necessarily $b\in \BMO_{\rm{rect}}$, i.e.
\begin{align*}
&\|b\|_{\BMO_{\rm{rect}}}:= \\
&\quad
\sup_{R=I\times J: \, \text{rectangles in }\R^{n}\times \R^m}\left(\frac 1{|R|}\int_R | b(x,y) -\langle b(\cdot, y)\rangle_I -\langle b(x, \cdot)\rangle_J +\langle b\rangle_R|^2\ud x\ud y\right)^{1/2}<\infty.
\end{align*}
Moreover,
\[
\|b\|_{\BMO_{\rm{rect}}}\lesssim \|[[b, T_1], T_2]\|_{L^2(\R^{n}\times \R^m)\to L^2(\R^{n}\times \R^m)}.
\]
\end{thm}
\noindent \textbf{Notations and conventions.}
Throughout this paper we use the notation $\langle f \rangle_E=\frac 1{|E|}\int_E f$, sometimes we may also write it as $f_E$. We write $A\lesssim B$ when there is some irrelevant constant $C>0$ such that $A\le C B$. When we write $A\sim B$ it means that $A\lesssim B$ and $B\lesssim A$ hold simultaneously. 

\vspace{0.3cm}

\noindent \textbf{Acknowledgements.}
This work is supported by the National Natural Science Foundation of China through
project numbers 12671125 and 12222114. We thank Prof. Tuomas Hyt\"onen for bringing \cite{WZ} and \cite{HytWu} to our attention. 

\vspace{0.3cm}

\noindent \textit{AI use statement} ChatGPT was used in order to understand $\inf_{p\in \mathcal P}\int_Q |p(b)|$ introduced in Section \ref{sec:ss2}, in particular Lemma \ref{lem2} is due to ChatGPT. All other mathematical content is solely due to the authors.

\section{Preliminaries}\label{sec:ss2}
This section is devoted to introducing some auxiliary results that will be needed in the proof of the main results. We begin with some notation about weights. 

We say a positive locally integrable function $w\in A_p(\R^n)$ with $1<p<\infty$ if 
\[
[w]_{A_p}:=\sup_{Q:\, \text{cubes in }\R^n}\, \langle w\rangle_Q  \langle w^{-\frac 1{p-1}}\rangle_Q^{p-1} <\infty.
\]When it is clear from the context we often suppress $\R^n$ and simply write $A_p$. 
Now given $\mu, \lambda\in A_p$, we define the so-called Bloom weight $\nu:= \mu^{1/p}\lambda^{-1/p}$. 
By definition it is easy to check that $\nu \in A_2$. 
Then the weighted BMO space $\BMO_\nu$ is defined as  
\[
\|b\|_{\BMO_\nu }:= \sup_{Q:\, \text{cubes in }\R^n} \frac 1{\nu(Q)}\int_Q |b-b_Q|.
\]
The following reverse H\"older property will be frequently used
\begin{prop}\cite[Lemma 2.5]{L22}\label{pp:rh}
Let $t_i>0$ and $w_i\in  A_{p_i}$ with $1<p_i<\infty$, $i=1,2$. Then there exists a constant $C$ depending only on $[w_i]_{A_{p_i}}$ and $t_i$ such that for any cube $Q$
\[
\langle w_1\rangle_Q^{t_1} \langle w_2\rangle_Q^{t_2}\le C\langle w_1^{t_1} w_2^{t_2}\rangle_Q. 
\]
\end{prop}

Now let us turn to the notations needed for our \emph{complex median method}.   Given a cube $Q$ and $b\in L_{\loc}^k(\R^n)$, let $p_0(z)=z^k+a_1z^{k-1}+\cdots+ a_k$ be a polynomial of degree $k$ so that 
\[
\frac 12\int_Q |p_0(b)|\le \inf_{p\in \mathcal P}\int_Q\abs{p(b)}\le \int_Q |p_0(b)|,
\]
where the infimum is taken over all polynomials of degree $k$ with the coefficient of $z^k$ equal to 1, that is
  \[
  \mathcal P:= \{p(z)=z^k+a_{1} z^{k-1}+\cdots+ a_k: a_j\in \C\}.
  \]
We may write 
\[
p_0(z)=\prod_{j=1}^k (z-\xi_j),
\]
and set
\[
E_i=\{y\in Q: \min_{1\le j\le k}|b(y)- \xi_j|= |b(y)-\xi_i|\}.
\]
Then $Q= \bigcup_i E_i$ and by the pigeon hole principle there exists some $i_0$ such that $$|E_{i_0}|\ge |Q|/k.$$
In below $\xi_{i_0}$ will play the role of `complex median'. For notational convenience we set $m_Q:=\xi_{i_0}$. 

We explain what is actually defined here. In the first order case, i.e. $k=1$, 
\[
\sup_Q \inf_c \frac 1{\nu(Q)}\int_Q |b-c|
\]
is exactly another definition of the $\BMO_\nu$ norm (see e.g. \cite{L22}). In this case we can simply take $p_0(z)= z-\langle b\rangle_Q$ and hence in this case $m_Q= \langle b\rangle_Q$. In the higher order case, the natural extension 
\[
\frac 1{\nu(Q)}\inf_{p\in \mathcal P}\int_Q|p(b)|
\]
turns to be important for us, which can be seen from the following characterization 
\begin{lem}\label{lem2}
  Let $b\in L^k_{\loc}$ and $Q$ be a cube, then there holds that
  \begin{equation}\label{eq:equivk}
   \inf_{p\in \mathcal P}\int_Q |p(b)|= \sup_{g\in \mathcal G_b} \left|\int_Q b^k g\right|,
  \end{equation}
  where 
  \[
  \mathcal G_b:=\{g: \|g\|_{L^\infty(Q)}\le 1, \int_Q g= \int_Q bg=\cdots=\int_Q b^{k-1}g=0\}.
  \]
 \end{lem}
\begin{proof}
Let $X=L^1(Q)$ and 
\[
V= \{a_1b^{k-1}+\cdots+ a_{k-1} b+a_k: a_j\in \C\}. 
\]
It is clear that $b^k\in X$ and $V$ is a finite dimensional closed subspace of $X$. We may set 
\[
d=  \inf_{p\in \mathcal P}\int_Q |p(b)|= \|b^k\|_{X/V}, 
\]
where $X/V$ is the usual quotient space endowed with the norm of $X$. By Hahn-Banach theorem,
there exists $\ell \in (X/V)^*$ such that $\|\ell\|\le 1$ and 
\[
\ell (b^k+V)= d.
\]
We may define $\Lambda: X \to \C$ by $\Lambda (u)= \ell(u+V)$. Hence 
\begin{equation}\label{eq:econv}
\|\Lambda\|\le 1,\quad \Lambda\big|_V=0,\quad \Lambda(b^k)=d.
\end{equation}
By Riesz's representation theorem there exists some $g_0\in L^\infty(Q)$ such that 
$$
\Lambda(u)=\int_Q u g_0,\qquad \|g_0\|_{L^\infty}=\|\Lambda\|\le 1.
$$
By \eqref{eq:econv} we have 
\[
d= \int_Q b^k g_0\le \sup_{g\in \mathcal G_b} \left|\int_Q b^k g\right|.
\]
On the other hand, it is trivial that 
\[
\sup_{g\in \mathcal G_b} \left|\int_Q b^k g\right| =
\inf_{p\in \mathcal P}\sup_{g\in \mathcal G_b} \left|\int_Q p(b) g\right|\le d.
\]
We are done.
 \end{proof}
Now with Lemma \ref{lem2} at hand, since $(z-m_Q)^k\in \mathcal P$ and recall the definition of $m_Q$,
roughly one can  at least build connection between 
\[
\frac 1{\nu(Q)}\int_Q |b-m_Q|^k\quad \text{and}\quad   \frac 1{\nu(Q)}\sup_{g\in \mathcal G_b} \left|\int_Q b^k g\right|.
\]
Recall for $g\in  \mathcal G_b$ one has 
\[
\int_Q b^k g= \int_Q (b(y)-b(x))^k g(y)\ud y= \frac 1{K(c_{\wt Q}, c_Q)}\int_Q (b(y)-b(x))^k K(c_{\wt Q}, c_Q)g(y)\ud y,
\]
then to relate it with $C_b^k(T)$ we should use Proposition \ref{prop:key}. Of course rigorously we need a careful case study. But this already clarifies how our new method works.

\section{Proof of Theorem \ref{thm}}\label{Sec:bl}
As mentioned, we shall prove the Bloom lower bound. 
That is, we will prove 
\[
\|b\|_{\BMO_{\nu^{1/k}}}\lesssim  \|C_b^k(T)\|_{L^p(\mu)\to L^p(\lambda)}^{1/k},
\]where $\nu=\mu^{1/p}\lambda^{-1/p}$ with $\mu, \lambda\in A_p$. 
In fact, by H\"older's inequality,
\begin{align*}
\int_Q |b- m_Q|&\le \Big(\int_Q |b- m_Q|^k\Big)^{1/k}|Q|^{1/{k'}}\\
&= \Osc_k^{\nu}(b, Q) \left(\frac{\nu(Q)}{|Q|}\right)^{1/k}|Q|
\lesssim \Osc_k^{\nu}(b, Q) \nu^{1/k}(Q),
\end{align*}
where 
\[
\Osc_k^{\nu}(b, Q):=\Big(\frac 1{\nu(Q)}\int_Q |b- m_Q|^k\Big)^{1/k}
\]
and
in the last step we have used Proposition \ref{pp:rh} (with $w_1=\nu$, $w_2=1$).  Then it suffices to prove 
\begin{equation}\label{eq:fc}
\sup_Q \Osc_k^{\nu}(b, Q) \lesssim \|C_b^k(T)\|_{L^p(\mu)\to L^p(\lambda)}^{1/k}.
\end{equation}
By Proposition \ref{prop:key}, associated with $Q$, there exists $\wt Q$ such that $\ell(Q)=\ell(\wt Q)$, $\dist(c_Q, c_{\wt Q})\sim A\ell(Q)$, and 
\[
|K(c_{\wt Q}, c_Q)|\sim A^{-n} |Q|^{-1}.
\]
As suggested by the usual median method, 
the value of $b$ on $\wt Q$ is important for us.  We split the problem into the following cases:
\begin{align*}
\text{Case I:}\qquad &\big|\big\{x\in \wt Q: |b(x)-m_Q|> C_k\Osc_k(b, Q) \big \}\big|\ge |\wt Q|/3\\
\text{Case II:}\qquad &\big|\big\{x\in \wt Q: |b(x)-m_Q|< C_k^{-1}\Osc_k(b, Q) \big \}\big|\ge |\wt Q|/3\\
\text{Case III:}\qquad &\big|\big\{x\in \wt Q: C_k^{-1}\Osc_k(b, Q)\le |b(x)-m_Q|\le C_k\Osc_k(b, Q) \big \}\big|\ge |\wt Q|/3,
\end{align*}
where $C_k$ is some large constant depending only on $k$, and $\Osc_k(b, Q):=\Osc_k^1(b, Q)$.
\subsection{The Case I}
Let 
\[
F:= \big\{x\in \wt Q: |b(x)-m_Q|> C_k\Osc_k(b, Q) \big \}.
\]
For any $x\in F$, we have 
\begin{align*}
\Big|\int_Q (b(x)-&b(y))^k \ud y\Big|\\
&= \Big|\int_Q (b(x)-m_Q+m_Q-b(y))^k \ud y\Big|\\
&\ge  \Big|\int_Q (b(x)-m_Q)^k \ud y\Big|- \sum_{i=0}^{k-1}\binom{k}{i}\Big|\int_Q (b(x)-m_Q)^i(m_Q-b(y))^{k-i} \ud y\Big|\\
&\ge |b(x)-m_Q|^k |Q|- \sum_{i=0}^{k-1}\binom{k}{i}|b(x)-m_Q|^i \Osc_k(b, Q)^{k-i} |Q|.
\end{align*}
Therefore, for sufficiently large $C_k$ we have 
\[
\Big|\int_Q (b(x)-b(y))^k \ud y\Big|\ge \frac 12  |b(x)-m_Q|^k |Q|.
\]
It follows that 
\begin{align*}
\frac 12|b(x)-m_Q|^k &\le \frac 1{|K(c_{\wt Q}, c_Q)| |Q|}\Big|\int_Q (b(x)-b(y))^k K(c_{\wt Q}, c_Q)\ud y\Big|\\
&\le \frac 1{|K(c_{\wt Q}, c_Q)||Q|}\Big|\int_Q (b(x)-b(y))^k K(x,y)\ud y\Big|\\
&\qquad+ \frac 1{|K(c_{\wt Q}, c_Q)||Q|}\Big|\int_Q (b(x)-b(y))^k (K(x,y)-K(c_{\wt Q}, c_Q)) \ud y\Big|\\
&\lesssim A^n |C_b^k(T)(1_Q)(x)| + A^n \int_Q |b(x)-b(y)|^k |K(x,y)-K(c_{\wt Q}, c_Q)| \ud y.
\end{align*}
Since 
\[
|b(x)-b(y)|^k\lesssim  |b(x)-m_Q|^k + |b(y)-m_Q|^k,
\]
by Proposition \ref{prop:key} we have that 
\begin{align*}
&\int_Q |b(x)-b(y)|^k |K(x,y)-K(c_{\wt Q}, c_Q)| \ud y\\
&\hspace{3cm} \lesssim\eps_A A^{-n} |b(x)-m_Q|^k+\int_Q |b(y)-m_Q|^k |K(x,y)-K(c_{\wt Q}, c_Q)| \ud y.
\end{align*}
Hence for sufficiently large $A$ we have
\begin{align*}
|b(x)-m_Q|^k\lesssim  A^n |C_b^k(T)(1_Q)(x)|+ A^n\int_Q |b(y)-m_Q|^k |K(x,y)-K(c_{\wt Q}, c_Q)| \ud y.
\end{align*}
Take the integration over $x\in F$ on both sides we have 
\begin{align*}
C_k^k \Osc_k(b,Q)^k |Q|/3 &\le \int_F |b(x)-m_Q|^k\\
&\lesssim  A^n \|C_b^k(T)(1_Q)\|_{L^1(F)}+  \eps_A  \Osc_k(b,Q)^k |Q|,
\end{align*}
where again we have used Proposition \ref{prop:key}. Now again by choosing $A$ to be sufficiently large we finally arrive at 
\begin{align*}
\Osc_k(b,Q)^k |Q|\lesssim_A  \|C_b^k(T)(1_Q)\|_{L^1(\wt Q)}&\le  \|C_b^k(T)(1_Q)\|_{L^p(\wt Q, \lambda)}
\lambda^{1-p'}(\wt Q)^{1/{p'}}\\
&\le \|C_b^k(T)\|_{L^p(\mu)\to L^p(\lambda)}\mu(Q)^{1/p}\lambda^{1-p'}(\wt Q)^{1/{p'}}.
\end{align*}
Now that $\lambda^{1-p'}\in A_{p'}$, so we have $\lambda^{1-p'}(\wt Q)\sim \lambda^{1-p'}( Q)$
and hence 
by Proposition \ref{pp:rh} we have 
\[
\mu(Q)^{1/p}\lambda^{1-p'}(\wt Q)^{1/{p'}}\lesssim \langle \mu^{1/p} \lambda^{(1-p')/{p'}}\rangle_Q |Q|= \nu(Q).
\]
Hence the claim \eqref{eq:fc} follows.

\subsection{The Case II} This case is completely similar as the previous case. In fact, let 
\[
\wt F= \big\{x\in \wt Q: |b(x)-m_Q|< C_k^{-1}\Osc_k(b, Q) \big \}.
\]
Set $$g=(\sign (b-m_Q))^k1_Q, \quad \text{where}\quad \sign (b-m_Q)=\frac{ \overline{b-m_Q}}{ |b-m_Q|}.$$ Observe that 
for any $x\in \wt F$
\begin{align*}
\Big|\int_Q (b(x)-&b(y))^k g(y)\ud y\Big|\\
&\ge \Big|\int_Q (m_Q-b(y))^k g(y)\ud y\Big|- \sum_{i=0}^{k-1}\binom{k}{i}\Big|\int_Q (b(x)-m_Q)^{k-i}(m_Q-b(y))^{i} \ud y\Big|\\
&\ge \int_Q |b(y)-m_Q|^k\ud y- \sum_{i=0}^{k-1}\binom{k}{i}|b(x)-m_Q|^{k-i} \Osc_k(b, Q)^i |Q|.
\end{align*}
Then for sufficiently large $C_k$, we have 
\[
\Big|\int_Q (b(x)-b(y))^k g(y)\ud y\Big|\ge \frac 12\int_Q |b(y)-m_Q|^k\ud y.
\]
It follows that 
\begin{align*}
\frac 12\int_Q |b(y)-m_Q|^k\ud y&\le \frac 1{|K(c_{\wt Q}, c_Q)| }\Big|\int_Q (b(x)-b(y))^k K(c_{\wt Q}, c_Q)g(y)\ud y\Big|\\
&\le \frac 1{|K(c_{\wt Q}, c_Q)|}\Big|\int_Q (b(x)-b(y))^k K(x,y)g(y)\ud y\Big|\\
&\qquad+ \frac 1{|K(c_{\wt Q}, c_Q)| }\Big|\int_Q (b(x)-b(y))^k (K(x,y)-K(c_{\wt Q}, c_Q))g(y) \ud y\Big|\\
&\lesssim A^n |Q|\cdot |C_b^k(T)(g)(x)|\\
&\qquad+ A^n |Q| \int_Q |b(x)-b(y)|^k | K(x,y)-K(c_{\wt Q}, c_Q)|\ud y,
\end{align*}where keep in mind that $x\in \wt F$. 
Since 
\begin{align*}
|b(x)-b(y)|^k \lesssim |b(x)-m_Q|^k + |b(y)-m_Q|^k,
\end{align*}
it follows from Proposition \ref{prop:key} that 
\begin{align*}
&\int_Q |b(x)-b(y)|^k | K(x,y)-K(c_{\wt Q}, c_Q)|\ud y\\
&\lesssim  \eps_A A^{-n}  |b(x)-m_Q|^k+ \int_Q |b(y)-m_Q|^k  | K(x,y)-K(c_{\wt Q}, c_Q)|\ud y\\
&\le  \eps_A A^{-n}  C_k^{-1}\Osc_k(b,Q)^k + \int_Q |b(y)-m_Q|^k  | K(x,y)-K(c_{\wt Q}, c_Q)|\ud y.
\end{align*}
Combining the above, we see that 
for sufficiently large $A$ we have
\begin{align*}
\Osc_k(b, Q)^k \lesssim A^n  |C_b^k(T)(g)(x)|+ A^n \int_Q |b(y)-m_Q|^k | K(x,y)-K(c_{\wt Q}, c_Q)|\ud y.
\end{align*}
Taking the $L^1(\wt F)$ norm on both sides and then by similar arguments as the previous case we will get the desired estimate. 

\subsection{The Case III}This is the most tricky case. And it is actually the critical reason why we introduce the polynomial $p_0$ in the very beginning. We begin with an additional assumption that 
\[
\inf_{p\in \mathcal P}\int_Q\abs{p(b)}= \sup_{g\in \mathcal G_b}\Babs{\int b^kg}\le \eta \int_Q |b-m_Q|^k,
\]
where $\eta$ is some small constant which will be determined very soon. In particular, this means 
\[
\int_Q \prod_{j=1}^k |b- \xi_j|\le 2\eta \int_Q |b-m_Q|^k,
\]
where recall that $|p_0(z)|=\prod_{j=1}^k |z- \xi_j|$ and 
\[
\frac 12\int_Q |p_0(b)|\le\inf_{p\in \mathcal P}\int_Q\abs{p(b)}\le \int_Q |p_0(b)|.
\]
Set 
\[
F_{i_0}=\{y\in E_{i_0}:\abs{b(y)-\xi_{i_0}}<\varepsilon \Osc_k(b, Q)\},
\]
where $\eps<1$ is a small constant which will be chosen later, also recall that $\xi_{i_0}=m_Q$ and $$|E_{i_0}|=|\{y\in Q: \min_{1\le j\le k}|b(y)- \xi_j|= |b(y)-\xi_{i_0}|\}|\ge |Q|/k.$$ 
Then 
\begin{align*}
\abs{E_{i_0}\setminus F_{i_0}}&\leq\frac{1}{\varepsilon^k  \Osc_k(b, Q)^k}\int_{E_{i_0}}\abs{b-\xi_{i_0}}^k\\
&\leq\frac{1}{\varepsilon^k  \Osc_k(b, Q)^k}\int_{E_{i_0}} \prod_{j=1}^k |b- \xi_j|\\
&\le \frac{2\eta }{\varepsilon^k  \Osc_k(b, Q)^k}\Osc_k(b, Q)^k |Q|\le \frac{2\eta k}{\eps^k }|E_{i_0}|.
\end{align*}
Hence, by taking  $\eta\ll\varepsilon^k$, it follows that $\abs{F_{i_0}}\sim\abs{E_{i_0}}\sim |Q|$. At this point let us define 
\[
 G :=\big\{x\in \wt Q: C_k^{-1}\Osc_k(b, Q)\le |b(x)-m_Q|\le C_k\Osc_k(b, Q) \big \},
\]
in our case we have $|G|\ge |Q|/3$. This means that for any $x\in G$ we can compute as the following
\begin{align*}
|b(x)-m_Q|^k |F_{i_0}| &=\Big|\int_{F_{i_0}}(b(x)-m_Q)^k\ud y\Big|\\
&\le \Big|\int_{F_{i_0}}(b(x)-b(y))^k\ud y\Big|+  \sum_{i=1}^{k}\binom{k}{i} \int_{F_{i_0}} |b(x)-b(y)|^{k-i}|b(y)-m_Q|^{i} \ud y.
\end{align*}
By the definition of $G$ and $F_{i_0}$ (and that $\xi_{i_0}=m_Q$), 
\[
|b(x)-b(y)|\le |b(x)-m_Q|+|b(y)-m_Q|\le (C_k+1)\Osc_k(b, Q), 
\]we have
\begin{align*}
\sum_{i=1}^{k}\binom{k}{i} \int_{F_{i_0}} |b(x)-b(y)|^{k-i}|b(y)-m_Q|^{i} \ud y\le  (C_k+1)^k 2^k \eps\Osc_k(b, Q)^k |F_{i_0}|.
\end{align*}
Let $\eps= (C_k+1)^{-2k}2^{-k-1}$, we finally obtain 
\[
\Osc_k(b, Q)^k \lesssim |F_{i_0}|^{-1}\Big|\int_{F_{i_0}}(b(x)-b(y))^k\ud y\Big|\sim |Q|^{-1}\Big|\int_{F_{i_0}}(b(x)-b(y))^k\ud y\Big|.
\]
The rest arguments are again quite similar as Case II, and we omit the details. We are left with the case 
\[
\inf_{p\in \mathcal P}\int_Q\abs{p(b)}= \sup_{g\in \mathcal G_b}\Babs{\int b^kg}\ge \eta \int_Q |b-m_Q|^k.
\]
In this case, we may choose a specific $g\in \mathcal G_b$ so that 
\begin{align*}
\int_Q |b-m_Q|^k &\le 2\eta^{-1}\Big|\int_Q b(y)^k g(y) \ud y\Big|\\
&= 2\eta^{-1}\Big|\int_Q (b(x)-b(y))^k g(y) \ud y\Big|.
\end{align*}
Again we are in similar situation as above, we are done. 
\begin{rem}
We remark that in the first order case the proof can be much simpler. Indeed, in the first order case one just needs to split the problem into two cases: 
\[
\big|\big\{x\in \wt Q: |b(x)-m_Q|>  \Osc_1(b, Q) \big \}\big|\ge |\wt Q|/2\\
\]
and 
\[
\big|\big\{x\in \wt Q: |b(x)-m_Q|\le \Osc_1(b, Q) \big \}\big|\ge |\wt Q|/2.
\]
The reason is that we do not need to use the binomial expansion.  
\end{rem}

\section{Proof of Theorem \ref{thm:m2}}
This section is devoted to proving Theorem \ref{thm:m2}. We begin with some notations. Let 
\[
\Osc_2(b,R):= \left(\frac 1{|R|}\int_R | b(x,y) -\langle b(\cdot, y)\rangle_I -\langle b(x, \cdot)\rangle_J +\langle b\rangle_R|^2\ud x\ud y\right)^{1/2}. 
\]
Fix a rectangle $R=I\times J$, using Proposition \ref{prop:key} for $K_1$ and $K_2$ will give us $\wt I$ and $\wt J$, respectively. Here $K_i$ stands for the kernel of $T_i$. 
Now in this bi-commutators setting the kernel will involve 
\[
B(x,y; u,v)= b(x,y)- b(x, v)-b(u, y)+b(u,v),\quad (x,y; u,v)\in \wt I \times \wt J \times I \times J.
\]
Hence understanding $B(x,y; u,v)$ will help us to formulate the median method in this setting. Recall that previously we only need to handle $b(x)-b(y)$, and we simply write $$b(x)-b(y)=(b(x)-m_Q)+(m_Q-b(y)).$$
So we have seen that the key is to decouple the variables, say $(x,y)$ and $(u,v)$. This inspires us to write 
\begin{align*}
B(x,y; u,v)&=\langle B(x,y; \,\cdot\, ,\,\cdot\,)\rangle_R + \langle B(\,\cdot\, ,\,\cdot\,; u,v)\rangle_{\wt R}
- B_{\wt I \times J}(x,v)- B_{I\times \wt J}(u, y)-\langle B\rangle_{\wt R\times R}\\
&=: \sum_{i=1}^5 B_i,
\end{align*}
where we have used the notation e.g.  
\[
B_{\wt I\times J}(x,v):= b(x,v) -\langle b(\cdot, v)\rangle_{\wt I} -\langle b(x, \cdot)\rangle_J +\langle b\rangle_{\wt I\times J}
\]
for $(x,v)\in \wt I\times J$. Careful readers may notice that we are not able to decouple the variables completely, but somehow it is already enough for our purpose.

Now we are in the position to introduce a suitable case study in the bi-commutator setting. Roughly speaking, since the rectangular product BMO norm is defined via $\Osc_2(b, R)$, which is an $L^2$ average, one cannot expect pointwise criteria for the case study. Instead, we care about the relation between 
\[
A_R:=\Big(\frac 1{|R||\wt R|}\iint_{\wt R \times R}|B(x,y;u,v)|^2\ud x\ud y\ud u \ud v\Big)^{1/2}
\]
and 
\[
\Sigma_R:=\Osc_2(b,R)+ \Osc_2(b,\wt I\times J)+\Osc_2(b, I\times \wt J).
\]

\subsection{The case $A_R> 2\Sigma_R$}
In this case we have 
\begin{equation}\label{eq:ar0}
A_R \le \sum_{i=1}^5\Big(\frac 1{|R||\wt R|}\iint_{\wt R \times R}|B_i|^2 \Big)^{1/2}. 
 \end{equation}
 First of all, we can write $|B_1|$ as 
 \begin{align*}
 |B_1|&= \Big|\frac 1{|R|}\int_R B(x, y; u, v) \ud u \ud v\Big|\\
 &=  \frac 1{|R| |K_1(c_{\wt I}, c_I)K_2(c_{\wt J}, c_J)|}\Big|\int_R B(x, y; u, v) K_1(c_{\wt I}, c_I)K_2(c_{\wt J}, c_J)\ud u \ud v\Big|\\
 &\lesssim A^{n+m} \Big| \int_R B(x, y; u, v) K_1(x, u)K_2(y, v)\ud u \ud v\Big|\\
 &\quad+ A^{n+m}\Big| \int_R B(x, y; u, v) (K_1(c_{\wt I}, c_I)K_2(c_{\wt J}, c_J)-K_1(x, u)K_2(y, v)) \ud u \ud v\Big|. 
 \end{align*}
By the regularity and size assumption of the kernel, we have that  
\begin{align*}
&|K_1(c_{\wt I}, c_I)K_2(c_{\wt J}, c_J)-K_1(x, u)K_2(y, v)|\\
&\le |K_1(c_{\wt I}, c_I)||K_2(c_{\wt J}, c_J)-K_2(y, v)|+ |K_2(y, v)| |K_1(c_{\wt I}, c_I)-K_1(x, u) |\\
&\lesssim \eps_A A^{-n-m}|R|^{-1}.
\end{align*}
Hence 
\begin{equation}\label{eq:eb1}
|B_1|\lesssim A^{n+m}\big|[ [b, T_1], T_2](1_R)\big|+ \eps_A \langle |B(x,y, u,v)|\rangle_R. 
\end{equation}
The estimate of $B_2$ will follow in a similar fashion. In fact, one just needs to interchange the role of $R$ and $\wt R$, and using $K_1(c_I, c_{\wt I})K_2(c_J, c_{\wt J})$ (this is where the symmetrically non-degeneracy is used). Thus we get 
\begin{equation}\label{eq:eb2}
|B_2|\lesssim A^{n+m}\big|[ [b, T_1], T_2](1_{\wt R})\big|+ \eps_A \langle |B(x,y, u,v)|\rangle_{\wt R}. 
\end{equation}
To deal with $B_3$, we simply notice that  
\begin{equation}\label{eq:eb3}
\Big(\frac 1{|R||\wt R|}\iint_{\wt R \times R}|B_3|^2 \Big)^{1/2}= \Osc_2(b,\wt I\times J).
\end{equation}
Likewise
\begin{equation}\label{eq:eb4}
\Big(\frac 1{|R||\wt R|}\iint_{\wt R \times R}|B_4|^2 \Big)^{1/2}= \Osc_2(b,I\times \wt J).
\end{equation}
Finally, we just view $\langle B\rangle_{\wt R\times R}$ as $\langle B_1\rangle_{\wt R}$, hence we have
\begin{equation}\label{eq:eb5}
|B_5|\lesssim  A^{n+m}\big\langle \big|[ [b, T_1], T_2](1_R)\big|\big\rangle_{\wt R}+ \eps_A \langle |B |\rangle_{\wt R\times R}.
\end{equation}
Substitute the estimates \eqref{eq:eb1}, \eqref{eq:eb2}, \eqref{eq:eb3}, \eqref{eq:eb4}, \eqref{eq:eb5} into \eqref{eq:ar0} and using H\"older's inequality we have 
\begin{align*}
A_R &\le C A^{n+m} |R|^{-1/2}\Big(2\|[ [b, T_1], T_2](1_R)\|_{L^2(\wt R)} +\|[ [b, T_1], T_2](1_{\wt R})\|_{L^2( R)}  \Big)\\
&\qquad+ 3C \eps_A A_R + \Osc_2(b,\wt I\times J)+ \Osc_2(b,I\times \wt J).
\end{align*}
Let $A$ be sufficiently large so that $3C \eps_A<1/4$, on the other hand 
\[
 \Osc_2(b,\wt I\times J)+ \Osc_2(b,I\times \wt J)\le \Sigma_R <\frac 12 A_R.
\]
Then we finally obtain that 
\begin{align*}
\Osc_2(b,R) &\le \Sigma_R < A_R/2 \\
&\lesssim A^{n+m} |R|^{-1/2}\Big(2\|[ [b, T_1], T_2](1_R)\|_{L^2(\wt R)} +\|[ [b, T_1], T_2](1_{\wt R})\|_{L^2( R)}  \Big)\\
&\lesssim_A  \|[ [b, T_1], T_2] \|_{L^2 \to L^2}.
\end{align*}

\subsection{The case $A_R\le 2\Sigma_R$}
By Riesz's representation theorem, there exists some $g\in L^2(R)$ with $\|g\|_{L^2(R)}=1$ such that 
\begin{align*}
\Osc_2(b, R)&= \frac 1{|R|^{1/2}} \int_R B_R (u,v) g(u, v)\ud u \ud v\\
&= \frac 1{|R|^{1/2}} \int_R B_R (u,v)  g_R(u,v) \ud u\ud v\\
&= \frac 1{|R|^{1/2}} \int_R b (u,v)  g_R(u,v) \ud u\ud v\\
&=  \frac 1{|R|^{1/2}} \int_R B(x,y; u,v) g_R(u,v) \ud u\ud v,
\end{align*}
where $(x, y)\in \wt R$ is arbitrary and 
\[
g_R(u,v) = g(u,v) -\langle g(\cdot, v)\rangle_I - \langle g(u, \cdot)\rangle_J + \langle g\rangle_R. 
\]
Then pretty much like the way we handled with $B_1$, we have 
\begin{align*}
\Osc_2(b, R)&\lesssim |R|^{1/2} A^{n+m}\big|[ [b, T_1], T_2](g_R)(x,y)\big| \\
&\qquad+ \eps_A |R|^{-1/2} \int_R |B(x, y; u, v)| |g_R(u,v)|  \ud u \ud v\\
&\le  |R|^{1/2} A^{n+m}\big|[ [b, T_1], T_2](g_R)(x,y)\big| + \eps_A |R|^{-1/2} \Big(\int_R | B(x, y; u, v)|^2 \ud u\ud v\Big)^{1/2}.
\end{align*}
Taking the $L^2(\wt R)$ norm on both sides we have 
\begin{equation}\label{eq:ec1}
\Osc_2(b, R)\lesssim  A^{n+m}\|[ [b, T_1], T_2](g_R)\|_{L^2(\wt R)}+ \eps_A A_R.
\end{equation}
Similarly, using the symmetrically non-degeneracy we have 
\begin{equation}\label{eq:ec2}
\begin{split}
\Osc_2(b, \wt I \times J)&\lesssim  A^{n+m}\|[ [b, T_1], T_2](g_{\wt I \times J})\|_{L^2(I \times \wt J)}+ \eps_A A_R,\\
\Osc_2(b, I \times \wt J)&\lesssim  A^{n+m}\|[ [b, T_1], T_2](g_{I \times \wt J})\|_{L^2(\wt I \times J)}+ \eps_A A_R,
\end{split}
\end{equation}
where $g_{\wt I \times J}$ and $g_{ I \times \wt J}$ are defined in a similar way as $g_R$, and  we have used the obvious observation 
\[
A_R= A_{\wt I \times J}= A_{I\times \wt J}.
\]
Combining \eqref{eq:ec1} and \eqref{eq:ec2} we obtain 
\begin{align*}
\Sigma_R &\le C A^{n+m} \Big( \|[ [b, T_1], T_2](g_R)\|_{L^2(\wt R)}+ \|[ [b, T_1], T_2](g_{\wt I \times J})\|_{L^2(I \times \wt J)}\\
&\qquad+\|[ [b, T_1], T_2](g_{I \times \wt J})\|_{L^2(\wt I \times J)}\Big)+ 3C \eps_A A_R.
\end{align*}
Let $A$ be sufficiently large so that 
\[
3C \eps_A A_R< \frac 14 A_R \le \frac 12 \Sigma_R, 
\]
we obtain 
\[
\Osc_2(b, R) \le \Sigma_R\lesssim_A   \|[ [b, T_1], T_2] \|_{L^2 \to L^2}.
\]
This completes the proof.

\end{document}